\documentclass{amsart}

\usepackage{lineno}
\usepackage{amssymb}
\usepackage{amsmath}
\usepackage{amsthm}
\usepackage{mathrsfs}
\usepackage{bbm}
\usepackage{cite}
\usepackage{color}
\usepackage{pgfplots}
\usepackage{enumerate}

\graphicspath{{figures/}}
\usepackage{colortbl}

\usepackage[colorlinks,hypertexnames=false]{hyperref}
\usepackage[msc-links]{amsrefs}

\usepackage{changes}

\usepackage{mathtools}
\usepackage{extarrows}
\usepackage{setspace}

\usepackage[misc]{ifsym}

\hypersetup{
	colorlinks=true,
	linkcolor=black,
	filecolor=black,
	urlcolor=black,
	citecolor=black,
}

\newtheorem{thm}{Theorem}[section]

\newtheorem{defn}[thm]{Definition}
\newtheorem{lem}[thm]{Lemma}
\newtheorem{prop}[thm]{Proposition}

\newtheorem{theorem}[thm]{Theorem}

\newtheorem{proposition}[thm]{Proposition}
\newtheorem{definition}[thm]{Definition}

\numberwithin{equation}{section}

\begin{document}

\title[Algebra of slice regular functions]{Algebra of slice regular functions\\ on non-symmetric domains in several quaternionic variables}
\author{Xinyuan Dou}
\email[Xinyuan Dou]{douxinyuan@ustc.edu.cn}
\address{Department of Mathematics, University of Science and Technology of China, Hefei 230026, China}
\address{Institute of Mathematics, AMSS, Chinese Academy of Sciences, Beijing 100190, China}
\author{Ming Jin\textsuperscript{\Letter}}
\email[Ming Jin]{mjin@must.edu.mo}
\address{Faculty of Innovation Engineering, Macau University of Science and Technology, Macau, China} 
\author{Guangbin Ren}
\email[Guangbin Ren]{rengb@ustc.edu.cn}
\address{Department of Mathematics, University of Science and Technology of China, Hefei 230026, China}
\author{Ting Yang}
\email[Ting Yang]{tingy@aqnu.edu.cn}
\address{School of Mathematics and Physics, Anqing Normal University, Anqing 246133, China}
\thanks{This work was supported by the China Postdoctoral Science Foundation (2021M703425), the NNSF of China (12171448), the Faculty Research Grants of the Macau University of Science and Technology (FRG-23-034-FIE), and Xiaomi Young Talents Program.}


\begin{abstract}
The primary objective of this paper is to establish an algebraic framework for the space of weakly slice regular functions over several quaternionic variables. We recently introduced a $*$-product that maintains the path-slice property within the class of path-slice functions. It is noteworthy that this $*$-product is directly applicable to weakly slice regular functions, as every slice regular function defined on a slice-open set inherently possesses path-slice properties.
Building on this foundation, we propose a precise definition of an open neighborhood for a path $\gamma$ in the path space $\mathscr{P}(\mathbb{C}^n)$. This definition is pivotal in establishing the holomorphism of stem functions. Consequently, we demonstrate that the $*$-product of two weakly slice regular functions retains its weakly slice regular nature. This retention is facilitated by holomorphy of stem functions and their relationship with weakly slice regular functions, providing a comprehensive algebraic structure for this class of functions.

\end{abstract}

\maketitle

\section{Introduction}

Slice analysis extends the theory of holomorphic functions in one complex variable to higher dimensions, offering a bridge to multifaceted mathematical disciplines. Initially introduced in the quaternions context by Gentili and Struppa in 2006 \cite{Gentili2006001, Gentili2007001}, this theory has since expanded, revealing that power series within non-commutative quaternion algebras qualify as slice regular functions, thus broadening the scope of slice analysis.

For those interested in the diverse applications and advancements in slice analysis, we recommend various sources: for geometric function theory, see \cite{Ren2017001,Ren2017002,Wang2017001}; for quaternionic Schur analysis, refer to \cite{Alpay2012001}; for insights into quaternionic operator theory, consult \cite{Alpay2015001,MR3887616,MR3967697,Gantner2020001,MR4496722}. Further extensions of slice analysis include its application to real Clifford algebras \cite{Colombo2009002}, octonions \cite{Gentili2010001}, real alternative $*$-algebras \cite{Ghiloni2011001}, and $2n$-dimensional Euclidean spaces \cite{Dou2023002}.

Two principal methodologies have emerged for exploring the several-variable aspect of slice analysis, both originating from the theories in one variable. The first, strongly slice regular functions, focuses on real alternative $*$-algebras in one variable \cite{Ghiloni2011001} and is characterized by its reliance on intrinsic stem functions and its applicability to axially symmetric domains. For several-variable extensions, see \cite{Colombo2012002} for the quaternionic context, and \cite{Ghiloni2012001} for real Clifford algebras and \cite{Ghiloni2020001} for real alternative $*$-algebras. The second approach, weakly slice regular functions, involves functions satisfying the Cauchy-Riemann equations within each complex plane slice, a concept originating from the initial definition of slice regular functions \cite{Gentili2007001}. Recent advancements in this area, such as the introduction of slice topology, have facilitated the study of non-axially symmetric domains, leading to comprehensive global theories \cite{Dou2023001}. For a broader understanding, refer to \cite{Dou2023001, Dou2021001} for general settings, and \cite{Dou2021001} for the octonions context.

This paper delves into the algebraic structure of weakly slice regular functions across multiple quaternionic variables, with a focus on identifying a suitable multiplication method that maintains slice regularity. Historically, two approaches to defining the $*$-product have been proposed: one based on the unique slice regular extension of the $*$-product between restricted functions on specific complex planes \cite{Gentili2008001,Colombo2009001}, and another defined via a corresponding multiplication on their stem functions \cite{Ghiloni2011001}. Both approaches, traditionally confined to axially symmetric slice domains, are equivalent in such contexts. However, the conventional extension theorem \cite{Colombo2009001} proves insufficient for non-axially symmetric slice domains. Consequently, we propose utilizing the second method for defining the $*$-product, fulfilling the necessary conditions. In \cite{Dou2023003}, we introduced a $*$-product that preserves the path-slice property, thus encompassing $*$-products of slice regular functions, as each slice regular function on a slice-open set is inherently path-slice. It's important to note that the criteria for a function to qualify as a stem function of a slice function are straightforward in axially symmetric slice domains. In these domains, the stem function is uniquely representable through the slice function, as per the representation formula. However, this becomes significantly more complex in non-axially symmetric slice domains. We addressed this challenge in \cite{Dou2023003}, offering auxiliary conditions to facilitate the process.

The primary focus of this paper is the slice regularity of the $*$-product of weakly slice regular functions. Central to this discussion is the holomorphy of stem functions. To apply differential operators to stem functions, we define an appropriate open neighborhood of a path $\gamma$ in the path space $\mathscr{P}(\mathbb{C}^n)$. In weak slice analysis, the corresponding stem function of a path-slice function is defined along paths, and the compatibility of stem function holomorphy with weak slice regular functions is a crucial consideration.

The paper is structured as follows: Section \ref{sec-preliminaries} revisits key results from \cite{Dou2023003}, including the relationship between path-slice functions and their stem functions, and the $*$-product between path-slice functions. Section \ref{sec-holo stem functions} introduces holomorphic stem functions, defining a suitable open neighborhood for paths in the path space $\mathscr{P}(\mathbb{C}^n)$ and establishing the holomorphy of stem functions corresponding to slice regular functions. Finally, Section \ref{sec-star product} demonstrates that the $*$-product of two slice regular functions retains slice regularity, thereby confirming that the space of weakly slice regular functions, when equipped with the $*$-product, forms an algebra.

\section{Preliminaries}\label{sec-preliminaries}

This section revisits certain key concepts in the study of slice regular functions and elaborates on the $*$-product of path-slice functions as presented in \cite{Dou2023003}. The quaternion algebra is denoted by $\mathbb{H}$. The set of imaginary units of quaternions is represented as:
\begin{equation*}
	\mathbb{S} := \{I \in \mathbb{H} : I^2 = -1\}.
\end{equation*}
For $n\in\mathbb{N}$, the $n$-dimensional weakly slice cone is defined as:
\begin{equation*}
	\mathbb{H}_s^n := \bigcup_{I \in \mathbb{S}} \mathbb{C}_I^n,
\end{equation*}
with slice topology
\begin{equation*}
	\tau_s(\mathbb{H}^n_s) := \{\Omega \subset \mathbb{H}^n_s : \Omega_I \in \tau(\mathbb{C}_I^n),\ \forall\ I \in \mathbb{S}\},
\end{equation*}
where $\mathbb{C}_I = \mathbb{R} + \mathbb{R}I$ and $\mathbb{C}_I^n = (\mathbb{C}_I)^n$
\begin{equation*}
	\Omega_I := \Omega \cap \mathbb{C}_I^n.
\end{equation*}
Open sets, connected sets, and paths in $\tau_s$ are termed slice-open sets, slice-connected sets, and slice-paths in $\Omega$, respectively.

For $I \in \mathbb{S}$, define:
\begin{equation*}
	\begin{split}
		\Psi_i^I:\quad\mathbb{C}^n\quad &\xlongrightarrow[\hskip1cm]{}\quad \mathbb{C}_I^n,
		\\   x+yi\ &\shortmid\! \xlongrightarrow[\hskip1cm]{}\ x+yI.
	\end{split}
\end{equation*}

The set of continuous paths with real-valued start points in $\mathbb{C}^n$ is denoted as:
\begin{equation*}
	\mathscr{P}(\mathbb{C}^n) := \{\gamma : [0,1] \rightarrow \mathbb{C}^n : \gamma \text{ is continuous with } \gamma(0) \in \mathbb{R}^n\}.
\end{equation*}
For a subset $\Omega \subset \mathbb{H}_s^n$, we define:
\begin{equation*}
	\mathscr{P}(\mathbb{C}^n,\Omega) := \left\{\delta \in \mathscr{P}(\mathbb{C}^n) : \exists\ I \in \mathbb{S}, \text{ such that } \delta^{I} \subset \Omega\right\},
\end{equation*}
where $\delta^I := \Psi_i^I(\delta)$ is a path in $\Omega_I$. For a fixed path $\gamma \in \mathscr{P}(\mathbb{C}^n)$, we define:
\begin{equation*}
	{\mathbb{S}(\Omega,\gamma)} := \{I \in \mathbb{S} : \gamma^{I} \subset \Omega \}.
\end{equation*}

\begin{definition}
	A function $f:\Omega \rightarrow \mathbb{H}$, for $\Omega \subset \mathbb{H}_s^n$, is termed \textit{\textbf{path-slice}} if it correlates with a function $F:\mathscr{P}(\mathbb{C}^n,\Omega) \rightarrow \mathbb{H}^{2 \times 1}$ such that:
	\begin{equation}\label{eq-fcg}
		f \circ \gamma^{I}(1) = (1, I) F(\gamma),
	\end{equation}
	for every $\gamma \in \mathscr{P}(\mathbb{C}^n, \Omega)$ and $I \in \mathbb{S}(\Omega, \gamma)$. The function $F$ is called a path-slice stem function of $f$. The set of path-slice functions defined on $\Omega$ is denoted as $\mathcal{PS}(\Omega)$, and $\mathcal{PSS}(f)$ denotes the set of path-slice stem functions of $f \in \mathcal{PS}(\Omega)$.
\end{definition}

\begin{definition}
	For $\Omega$ in the slice topology $\tau_s(\mathbb{H}_s^n)$, a function $f:\Omega \rightarrow \mathbb{H}$ is considered \textit{\textbf{(weakly) slice regular}} if, for every $I \in \mathbb{S}$, the restriction $f_I := f|{\Omega_I}$ is (left $I$-) holomorphic. This means $f_I$ is real differentiable and satisfies:
	\begin{equation*}
		\frac{1}{2}\left(\frac{\partial}{\partial x_\ell} + I \frac{\partial}{\partial y_\ell}\right) f_I(x + yI) = 0
	\end{equation*}
on $\Omega_I$ 
	for each $\ell = 1, 2, ..., d$. The class of weakly slice regular functions on $\Omega$ is denoted as $\mathcal{SR}(\Omega)$.
\end{definition}

According to \cite[Corollary 5.9]{Dou2023002}, weakly slice regular functions qualify as path-slice.

\begin{definition}
	A subset $\Omega \subset \mathbb{H}_s^n$ is defined as \textit{\textbf{real-path-connected}} if, for each point $q \in \Omega$, there exists a path $\gamma \in \mathscr{P}(\mathbb{C}^n, \Omega)$ and an element $I \in \mathbb{S}(\Omega, \gamma)$ such that $\gamma^I(1) = q$.
\end{definition}

Given a subset $\Omega \subset \mathbb{H}_s^n$, denote
\begin{equation}\label{eq-mp*2}
	\mathscr{P}_*^2(\mathbb{C}^n, \Omega) := \{(\alpha, \beta) \in [\mathscr{P}(\mathbb{C}^n, \Omega)]^2 : \alpha(1) = \beta(1) \}.
\end{equation}

\begin{definition}
	For subsets $\Omega_1, \Omega_2 \subset \mathbb{H}_s^n$, $\Omega_2$ is termed \textit{\textbf{$\Omega_1$-stem-preserving}} if it satisfies the following conditions:
	\begin{enumerate}[\upshape (i)]
		\item $\left|\mathbb{S}(\Omega_2, \gamma)\right| \geqslant  2$, for every $\gamma \in \mathscr{P}(\mathbb{C}^n, \Omega_1)$.
		\item $\left|\mathbb{S}(\Omega_2, \alpha) \cap \mathbb{S}(\Omega_2, \beta)\right| \neq 1$, for each $(\alpha, \beta) \in \mathscr{P}_*^2(\mathbb{C}^n, \Omega_1)$.
	\end{enumerate}
\end{definition}

Based on [1, Proposition 3.7], the subsequent definition is deemed well-defined:

\begin{definition}
	For a real-path-connected subset $\Omega_1 \subset \mathbb{H}_s^n$, an $\Omega_1$-stem-preserving subset $\Omega_2 \subset \mathbb{H}_s^n$, and a path-slice function $f:\Omega_2 \rightarrow \mathbb{H}$, the following mapping 
	\begin{equation}\label{eq-F def}
		\begin{split}
			F_{\Omega_1}^f :  \mathscr{P}(\mathbb{C}^n, \Omega_1) &\xlongrightarrow[\hskip1cm]{} \mathbb{H}^{2 \times 1}
			\\     \gamma &\shortmid\!\xlongrightarrow[\hskip1cm]{}G|_{\mathscr{P}(\mathbb{C}^n, \Omega_1)},
		\end{split}
	\end{equation}
	is well defined and independent of the chosen $G \in \mathcal{PSS}(f)$.
\end{definition}

\begin{proposition}\cite[Proposition 3.3]{Dou2023003}.
	Given a domain $\Omega\subset\mathbb{H}_s^n$, a function $f\in\mathcal{PS}(\Omega)$, a path $\gamma\in\mathscr{P}(\mathbb{C}^n,\Omega)$, and a path-slice stem function $F$ associated with $f$, for any distinct $I, J\in\mathbb{S}(\Omega,\gamma)$, it holds that
	\begin{equation}\label{eq-fgbp}
		F(\gamma) = \begin{pmatrix}
			1 & I \\
			1 & J
		\end{pmatrix}^{-1} \begin{pmatrix}
			f\circ\gamma^I(1) \\
			f\circ\gamma^J(1)
		\end{pmatrix}.
	\end{equation}
\end{proposition}

Define the function $\mathfrak{I}: \mathbb{H}_s^n \rightarrow \mathbb{S}\cup{0}$ for $q=(q_1,...,q_n)\in\mathbb{H}_s^n$ by
\begin{equation*}
	\mathfrak{I}(q) = \begin{cases}
		0, & \text{if } q\in\mathbb{R}^n, \\
		\displaystyle\frac{q_\imath-Re(q_\imath)}{|q_\imath-Re(q_\imath)|}, & \text{otherwise},
	\end{cases}
\end{equation*}
where $\imath\in\{1,...,n\}$ is the smallest index such that $q_\imath\notin\mathbb{R}$.

From \cite[Proposition 3.11]{Dou2023003}, we define:

\begin{definition}
	For a real-path-connected domain $\Omega_1\subset\mathbb{H}_s^n$ and an $\Omega_1$-stem-preserving domain $\Omega_2\subset\mathbb{H}_s^n$, if $f:\Omega_2\rightarrow\mathbb{H}$ is path-slice, then
	\begin{equation}\label{eq-mathscr f def}
		\mathscr{F}_{\Omega_1}^f(q) = \begin{cases}
			F_{\Omega_1}^f(\gamma), & \text{if } q\notin\mathbb{R}^n, \\
			(f(q), 0)^T, & \text{otherwise},
		\end{cases}
	\end{equation}
	is well-defined for $q\in\Omega_1$, where $\gamma\in\mathscr{P}(\mathbb{C}^n,\Omega_1)$ satisfies $\gamma^{\mathfrak{I}(q)}\subset\Omega_1$ and $\gamma^{\mathfrak{I}(q)}(1)=q$.
\end{definition}

\begin{defn}
	Given real-path-connected $\Omega_1\subset\mathbb{H}_s^n$, an $\Omega_1$-stem-preserving $\Omega_2\subset\mathbb{H}_s^n$, and functions $f\in\mathcal{PS}(\Omega_1)$, $g\in\mathcal{PS}(\Omega_2)$, the $*$-product of $f$ and $g$ is defined as
	\begin{equation}\label{eq-starproduct def}
		f*g := (f,\mathfrak{I}f) \mathscr{F}_{\Omega_1}^{g}:\Omega_1\rightarrow\mathbb{H}.
	\end{equation}
\end{defn}

For $p:=(p_1,p_2)^T, q:=(q_1,q_2)^T\in \mathbb{H}^{2\times 1}$, denote the product $p*q$ as
\begin{equation*}
	p*q := (p_1\mathbb{I} + p_2\sigma)(q_1\mathbb{I} + q_2\sigma)e_1,
\end{equation*}
where $\mathbb{I}$, $\sigma$, and $e_1$ are defined as
\begin{equation*}
	\mathbb{I} := \begin{pmatrix}
		1 & 0 \\
		0 & 1
	\end{pmatrix},\
	\sigma := \begin{pmatrix}
		0 & -1 \\
		1 & 0
	\end{pmatrix},\
	e_1 := \begin{pmatrix}
		1 \\
		0
	\end{pmatrix}.
\end{equation*}

For functions $F, G: \mathscr{P}(\mathbb{C}^n,\Omega)\rightarrow\mathbb{H}^{2\times 1}$ for a domain $\Omega\subset\mathbb{H}_s^n$, define
\begin{equation*}
	F*G(\gamma) := F(\gamma) * G(\gamma),\qquad\forall\ \gamma\in\mathscr{P}(\mathbb{C}^n,\Omega).
\end{equation*}

\begin{lem} \cite[Lemma 4.2]{Dou2023003}. Let $\Omega_1 \subset\mathbb{H}_s^n$  
be  real-path-connected, and $\Omega_2 \subset\mathbb{H}_s^n$   be $\Omega_1$-stem-preserving domain.   For $c\in\mathbb{H}$, $g\in\mathcal{PS}(\Omega_2)$, a path $\gamma\in\mathscr{P}(\mathbb{C}^n,\Omega_1)$, and $I\in\mathbb{S}(\Omega_1,\gamma)$ where $q:=\gamma^I(1)$, it follows that
	\begin{equation}\label{eq-overline}
		(c,Ic) F_{\Omega_1}^{g}(\gamma) = (c,-Ic) F_{\Omega_1}^{g}(\overline{\gamma}) = \left(c,\mathfrak{I}(q)c\right) \mathscr{F}_{\Omega_1}^{g}(q).
	\end{equation}
\end{lem}

\begin{definition}
	A domain $\Omega\subset\mathbb{H}_s^n$ is termed self-stem-preserving if it is real-path-connected and $\Omega$-stem-preserving.
\end{definition}

\begin{theorem}\label{tm-aura}\cite[Theorem 5.3]{Dou2023003}.
	If a domain $\Omega\subset\mathbb{H}_s^n$ is self-stem-preserving, then the structure $(\mathcal{PS}(\Omega),+,*)$ forms an associative unitary real algebra.
\end{theorem}

\section{Holomorphic stem functions}\label{sec-holo stem functions}
In this section, we explore the concept of holomorphy for path-dependent stem functions.

We discover that a path-slice stem function 
$F$ derived from a path-slice function 
$f$ is not inherently holomorphic. However, the specific restriction
 $F_{\Omega_1}^f$ of $F$ exhibits holomorphy.

For any  two points $z,w$ in the complex space $\mathbb{C}^n$, we define a standard path connecting $z$ to $w$. This is a  function 
 $\mathcal{L}_z^w: [0, 1] \to \mathbb C^n$ given  by the formula
 $$
\mathcal{L}_z^w(t)=(1-t)z+tw.$$

For a path $\gamma\in\mathscr{P}(\mathbb{C}^n)$ and a positive radius 
$r$, $B_{\mathbb{C}^n}(\gamma(1),r)$ represents a ball in  $\mathbb{C}^n$ centered at the endpoint of the path  $\gamma(1)$ with radius $r$.

  The ball  in  the path space  $\mathscr{P}(\mathbb{C}^n)$ is defined as the set of all paths one can form by taking the standard path
    $\mathcal{L}_{\gamma(1)}^z$ from the end point of $\gamma$ to the point $z$  inside the ball
  $B_{\mathbb{C}^n}(\gamma(1),r)$ in $\mathbb C^n$.
  Essentially, this process  creates  new paths in  $\mathscr{P}(\mathbb{C}^n)$
    by extending 
$\gamma$ to reach any point within a radius 
 $r$ from its endpoint, and each of these new paths starts from the origin of 
 $\gamma$ and ends at some point within this spherical region in $\mathbb C^n$.
This concept is mathematically expressed as
\begin{equation}\label{eq-BP def}
	B_{\mathscr{P}(\mathbb{C}^n)}(\gamma,r):= 
		\left\{\gamma\circ\mathcal{L}_{\gamma(1)}^{z}:z\in B_{\mathbb{C}}(\gamma(1),r)\right\}.
\end{equation}

Let $\Omega\subset\mathbb{H}_s^n$ and  $\gamma\in\mathscr{P}(\mathbb{C}^n)$. We consider  
  a subset  $\mathbb{S}'$ of $\mathbb{S}(\Omega,\gamma)$, 
which could be specific selections or categories within a larger set that are related to the region 
 $\Omega$ and the path $\gamma$.

We need to introduce some concepts which are about finding the largest possible balls  in different spaces, and how these balls fit within certain regions.

\begin{itemize}
	\item $r_{\gamma, \Omega}$: This represents the supremum of radii 
	$r$ for which the ball in the path space centered at 
$\gamma$ with radius 
$r$ is entirely contained within the path space related to 
$\Omega$. Geometrically, this is finding the largest possible radius of a ball  in the path space, centered at 
$\gamma$, that still fits entirely within the region defined by 
$\Omega$.

\item 	$r_{\gamma,\scriptscriptstyle\Omega}$:
   This is the supremum of radii 
	$r$ for which the ball centered at the endpoint of 
 $\gamma$ (in the image space of 
	$\gamma$) with radius 
$r$ is contained within $\Omega_I$
	for all selections 
$I$ in 
$\mathbb S'$.
  Geometrically, it is about finding the largest radius of a ball  in the image space that fits within the 
	regions  $\Omega_I$
	 for every category in 
$\mathbb S'$.
	
 \item  	$r_{\gamma,\scriptscriptstyle\Omega}^2$: This represents the supremum of 
	$r_{\gamma,\scriptscriptstyle\Omega}^{\mathbb{S}''}$
 	over all subsets 
$\mathbb S^{''}$
	of 
$\mathbb S(\Omega, \gamma)$ that have at least two elements. Geometrically, this is finding the largest ball  (in terms of radius) that fits within the $\Omega_I$
	regions for every possible subset of 
$\mathbb S(\Omega, \gamma)$ that contains two or more elements.
	\end{itemize}

These concepts can be expressed as 
\begin{eqnarray*}
	r_{\gamma,\scriptscriptstyle\Omega}&:=& \sup\left\{r\in[0,+\infty):B_{\mathscr{P}(\mathbb{C}^n)}(\gamma,r)\subset\mathscr{P}(\mathbb{C}^n,\Omega)\right\},
\\
	r_{\gamma,\scriptscriptstyle\Omega}^{\mathbb{S}'}&:=& \sup\left\{r\in[0,+\infty):B_I\left(\gamma^I(1),r\right)\subset\Omega_I,\ \forall\ I\in\mathbb{S}'\right\}.
\\
	r_{\gamma,\scriptscriptstyle\Omega}^2 &:=& \sup\left\{r_{\gamma,\scriptscriptstyle\Omega}^{\mathbb{S}''}:\mathbb{S}''\subset\mathbb{S}(\Omega,\gamma)\mbox{ with }|\mathbb{S}''|\geqslant  2\right\}.
\end{eqnarray*}

Under certain conditions, these constants remain positive, which proves to be valuable for practical applications.

\begin{lem} For any  set $\Omega\in\tau_s(\mathbb{H}_s^n)$, it holds that 
	\begin{equation*}
		r_{\gamma,\scriptscriptstyle\Omega}>0
	\end{equation*}
for every path $\gamma$ in the set $\mathscr{P}(\mathbb{C}^n,\Omega).$
\end{lem}

\begin{proof}
	Assume $\gamma$ is a path in the path space $\mathscr{P}(\mathbb{C}^n, \Omega)$ and let $I$ be an element of $\mathbb{S}(\Omega, \gamma)$. Given the property that $\Omega$ is a slice-open set, we can find a positive radius $r$ such that the ball $B_I(\gamma^I(1), r)$ is entirely contained within the subset $(\Omega)_I$ of $\Omega$.

	We now consider any \( z \in B_{\mathbb{C}}(\gamma(1), r) \). For this choice of \( z \), we derive the following implications 
	\begin{equation*}
		\left(\mathcal{L}_{\gamma(1)}^z\right)^I \subseteq B_I(\gamma^I(1), r) \subseteq \Omega_I.
	\end{equation*}
	Subsequently, we deduce 
	\begin{equation*}
		\left(\gamma \circ \mathcal{L}_{\gamma(1)}^z\right)^I = \gamma^I \circ \left(\mathcal{L}_{\gamma(1)}^z\right)^I \subseteq \Omega_I.
	\end{equation*}
	This leads to the inference that 
	\begin{equation*}
		\gamma \circ \mathcal{L}_{\gamma(1)}^z \in \mathscr{P}(\mathbb{C}^n, \Omega).
	\end{equation*}
Since  $B_{\mathscr{P}(\mathbb{C}^n)}(\gamma, r)$ is a subset of $\mathscr{P}(\mathbb{C}^n, \Omega)$,  it follows   that $r_{\gamma, \scriptscriptstyle\Omega} \geqslant r > 0$.
	 \end{proof}

 \begin{lem}
 	Suppose $\Omega_1$ is a real-path-connected subset of $\tau_s(\mathbb{H}_s^n)$, and $\Omega_2$, also in $\tau_s(\mathbb{H}_s^n)$, is  $\Omega_1$-stem-preserving. Then, for every path $\gamma$ in the path space $\mathscr{P}(\mathbb{C}^n, \Omega_1)$, the following inequality holds:
 	\begin{equation}\label{eq-rgammaOmega2}
 		r_{\gamma,\scriptscriptstyle\Omega_2}^2 > 0.
 	\end{equation}
 \end{lem}
 
 \begin{proof}
 	Assume $\gamma$ is a path in the path space $\mathscr{P}(\mathbb{C}^n, \Omega_1)$. Given that $\Omega_2$ in $\tau_s(\mathbb{H}_s^n)$ is  $\Omega_1$-stem-preserving, the set $\mathbb{S}(\Omega_2, \gamma)$ contains at least two distinct elements. Let $I$ and $J$ be two such distinct elements in $\mathbb{S}(\Omega_2, \gamma)$. Since $\Omega_2$ is slice-open, there exists a positive radius $r$ satisfying 
 	\begin{equation*}
 		B_L(\gamma^L(1), r) \subset (\Omega_2)_L,
 	\end{equation*}
 for every   $L \in \{I, J\}.$ 
 	This leads to the conclusion that
 	\begin{equation*}
 		0 < r \leqslant r^{\{I, J\}}_{\gamma, \Omega_2} \leqslant  r^2_{\gamma, \Omega_2}.
 	\end{equation*}
 \end{proof}

 \begin{lem}
 	Consider $\Omega_1$ as a subset of $\tau_s(\mathbb{H}_s^n)$ that is real-path-connected, and $\Omega_2$ also in $\tau_s(\mathbb{H}_s^n)$, which is $\Omega_1$-stem-preserving. Let $\gamma$ be a path in the path space $\mathscr{P}(\mathbb{C}^n, \Omega_1)$ and   $r_1$ be arbitrary from the interval $\left(0, \min\{r_{\gamma,\scriptscriptstyle\Omega_2}^2, r_{\gamma,\scriptscriptstyle\Omega_1}\}\right)$. Then, there exist distinct elements $I, J$ in $\mathbb{S}(\Omega_2, \gamma)$  such that for every path $\beta$ in the ball $B_{\mathscr{P}(\mathbb{C}^n)}(\gamma, r_1)$, which is a subset of $\mathscr{P}(\mathbb{C}^n, \Omega_1)$, the paths  $\beta^I$ and $\beta^J$ are contained within $\Omega_2$, i.e., 
 	\begin{equation}\label{eq-beta in Omega2}
 		\beta^I, \beta^J \subset \Omega_2,\qquad\forall\ \beta\in B_{\mathscr{P}(\mathbb{C}^n)}(\gamma, r_1).
 	\end{equation}
 \end{lem}

 \begin{proof}
 	Given that $r_1 < r_{\gamma,\scriptscriptstyle\Omega_1}$, we have the inclusion
 	\begin{equation*}
 		B_{\mathscr{P}(\mathbb{C}^n)}(\gamma, r_1) \subseteq \mathscr{P}(\mathbb{C}^n, \Omega_1).
 	\end{equation*}
 	From the condition $r_1 < r_{\gamma,\scriptscriptstyle\Omega_2}^2$, there exists a subset $\mathbb{S}'$ of $\mathbb{S}(\Omega_2, \gamma)$ with $|\mathbb{S}'|\ge 2$ satisfying
 	\begin{equation*}
 		B_L(\gamma^L(1), r_1) \subseteq (\Omega_2)_L
 	\end{equation*}
 for each  $ L \in \mathbb{S}'.$
 	Select distinct elements $I, J \in \mathbb{S}'$ and consider any path $\beta \in B_{\mathscr{P}(\mathbb{C}^n)}(\gamma, r_1)$. By the definition of $\beta$, we have $\beta = \gamma \circ \mathcal{L}_{\gamma(1)}^{\beta(1)}$. Then, for $L \in \{I, J\}$, the following holds:
 	\begin{equation*}
 		\left(\mathcal{L}_{\gamma(1)}^{\beta(1)}\right)^L \subseteq B_L(\gamma^L(1), r) \subseteq (\Omega_2)_L.
 	\end{equation*}
 	Consequently, this implies
 	\begin{equation*}
 		\beta^L = \left(\gamma \circ \mathcal{L}_{\gamma(1)}^{\beta(1)}\right)^L = \gamma^L \circ \left(\mathcal{L}_{\gamma(1)}^{\beta(1)}\right)^L \subseteq (\Omega_2)_L,
 	\end{equation*}
for $ L \in \{I, J\}$. 
 	Thus, the condition \eqref{eq-beta in Omega2} is satisfied.
 \end{proof}

For any  $\gamma\in\mathscr{P}(\mathbb{C}^n)$ and $r>0$, we introduce a map 
\begin{equation}\label{eq-mathscr f def}
	\begin{split}
		\mathscr{L}_{\gamma}:   B_{\mathbb{C}}(\gamma(1),r)  \xlongrightarrow[\hskip0.5cm]{}  B_{\mathscr{P}(\mathbb{C}^n)}(\gamma,r), 
	\end{split}
\end{equation}
defined by
$$ \mathscr{L}_{\gamma}(z)=  \gamma\circ\mathcal{L}_{\gamma(1)}^z. $$

The transformation $\gamma\circ\mathcal{L}_{\gamma(1)}^z$ geometrically represents the extension of the original path 
 to reach the point $z$.    It illustrates how a point-wise approach (selecting a point $z$)
  can lead to a new understanding or interpretation in a path-dependent framework (creating a new path in  $\mathscr{P}(\mathbb{C}^n)$).
  This highlights the interconnectedness and complementary nature of point-wise and path-dependent perspectives in slice  analysis.

\begin{proposition}  	Let $\Omega_1\in\tau_s(\mathbb{H}_s^n)$ be real-path-connected, $\Omega_2\in\tau_s(\mathbb{H}_s^n)$ be $\Omega_1$-stem-preserving, $\gamma\in\mathscr{P}(\mathbb{C}^n,\Omega_1)$ and $f\in\mathcal{SR}(\Omega_2)$.
 Then, the following equation holds:
	\begin{eqnarray}
		F_{\Omega_1}^f\circ\mathscr{L}_{\gamma}(z)=\begin{pmatrix}
			1&I\\ 1& J
		\end{pmatrix}^{-1}\begin{pmatrix}
			f(z^I)\\ f(z^J)
		\end{pmatrix}
	\end{eqnarray}
	for any $z$ in the complex ball $B_{\mathbb{C}^n}\left(\gamma(1),r_1\right)$, where $r_1\in\left(0,\min\left\{r_{\gamma,\scriptscriptstyle\Omega_2}^2,r_{\gamma,\scriptscriptstyle\Omega_1}\right\}\right)$ and  $I,J\in\mathbb{S}(\Omega_2,\gamma)$ with $I\neq J$ and \eqref{eq-beta in Omega2} holds.
\end{proposition}

\begin{proof}
	Suppose we take any $z$ from $B_{\mathbb{C}}\left(\gamma(1),r_1\right)$. Define $\beta$ as $$\mathscr{L}_{\gamma}(z) = \gamma \circ \mathcal{L}_{\gamma(1)}^z$$ and let $F\in \mathcal{PSS}(f)$. It follows that:
	\begin{equation*}
		\beta(1) =  \left[\gamma \circ \mathcal{L}_{\gamma(1)}^z \right] (1)= z.
	\end{equation*}
	Given \eqref{eq-beta in Omega2}, $I$ and $J$ are elements of $\mathbb{S}(\Omega_2,\beta)$. Referring to \eqref{eq-F def} and \eqref{eq-fgbp}, we can deduce 
	\	\begin{equation*}
	\begin{split}
		F_{\Omega_1}^f\circ\mathscr{L}_{\gamma}(z)=F_{\Omega_1}^f(\beta)=F(\beta)
		=&\begin{pmatrix}1&I\\1&J\end{pmatrix}^{-1}
		\begin{pmatrix}f\circ\beta^I(1)\\f\circ\beta^J(1)\end{pmatrix}
		\\=&\begin{pmatrix}1&I\\1&J\end{pmatrix}^{-1}
		\begin{pmatrix}f\left(z^I\right)\\f\left(z^J\right)\end{pmatrix}.
	\end{split}
\end{equation*}
This completes the proof. \end{proof}

We are now in a position to define the notion of a function being holomorphic over a space of paths.

\begin{defn}
	Let $\Omega\subset\mathbb{H}_s^n$ and $\gamma\in\mathscr{P}(\mathbb{C}^n)$. A function $F:\mathscr{P}(\mathbb{C}^n,\Omega)\rightarrow\mathbb{H}^{2\times 1}$ is called \textit{\textbf{holomorphic}} at $\gamma$, if there is $r>0$ such that $B_{\mathscr{P}(\mathbb{C}^n)}(\gamma,r)\subset\mathscr{P}(\mathbb{C}^n,\Omega)$ and
	\begin{equation}\label{eq-holo def}
		\frac{1}{2}\left(\frac{\partial}{\partial x_\ell}+\sigma\frac{\partial}{\partial y_\ell}\right)\left(F\circ\mathscr{L}_{\gamma}\right)(x+yi)=0,
	\end{equation}
	for each $x+yi\in B_{\mathbb{C}}(\gamma(1),r)$ and $\ell\in\{1,...,n\}$.

	Moreover, we call $F$ is holomorphic in $U\subset \mathscr{P}(\mathbb{C}^n,\Omega)$ if $F$ is holomorphic at each $\gamma\in U$. 
\end{defn}

\begin{prop}\label{pr-FOmega1f holo}
	Let $\Omega_1\in\tau_s(\mathbb{H}_s^n)$ be real-path-connected, $\Omega_2\in\tau_s(\mathbb{H}_s^n)$ be $\Omega_1$-stem-preserving, and $f\in\mathcal{SR}(\Omega_2)$. Then $F_{\Omega_1}^f$ is holomorphic in $\mathscr{P}(\mathbb{C}^n,\Omega).$
\end{prop}

\begin{proof}
	For each path $\gamma$ in $\mathscr{P}(\mathbb{C}^n,\Omega_1)$, let us select $r_1$ from the interval $ (0,\min(r_{\gamma,\scriptscriptstyle\Omega_2}^2,r_{\gamma,\scriptscriptstyle\Omega_1})) $. We choose two distinct elements $I, J$ from $\mathbb{S}(\Omega_2,\gamma)$ such that equation \eqref{eq-beta in Omega2} is satisfied.
	
	Observe the following identity:
	\begin{equation*}
		\sigma\begin{pmatrix} 1 & I\\ 1 & J \end{pmatrix}^{-1}
		=\begin{pmatrix} 1 & I\\ 1 & J \end{pmatrix}^{-1}
		\begin{pmatrix} I & \\ & J \end{pmatrix}.
	\end{equation*}
	
	Upon direct computation, we arrive at
	\begin{equation*}
		\begin{aligned}
			&\frac{1}{2}\left(\frac{\partial}{\partial x_\ell}+\sigma\frac{\partial}{\partial y_\ell}\right)\left(F_{\Omega_1}^f\circ\mathscr{L}{\gamma}\right)(z)
			\\=&\frac{1}{2}\left(\frac{\partial}{\partial x\ell}+\sigma\frac{\partial}{\partial y_\ell}\right)
			\begin{pmatrix} 1&I\\1& J \end{pmatrix}^{-1}
			\begin{pmatrix} f(z^I)\\ f(z^J) \end{pmatrix}
			\\=&\begin{pmatrix} 1&I\\1& J \end{pmatrix}^{-1}
			\left[\frac{1}{2}\left(\frac{\partial}{\partial x_\ell}+
			\begin{pmatrix} I\\ & J \end{pmatrix}\frac{\partial}{\partial y_\ell}\right)\right]
			\begin{pmatrix} f(z^I)\\ f(z^J) \end{pmatrix}
			\\=&\begin{pmatrix} 1&I\\1& J \end{pmatrix}^{-1}
			\begin{pmatrix}\frac{1}{2}\left(\frac{\partial}{\partial x_\ell}+I\frac{\partial}{\partial y_\ell}\right) f(z^I)\\  \frac{1}{2}\left(\frac{\partial}{\partial x_\ell}+J\frac{\partial}{\partial y_\ell}\right) f(z^J) \end{pmatrix}
			\\=&\begin{pmatrix} 1&I\\ 1& J \end{pmatrix}^{-1} \begin{pmatrix}0\\ 0\end{pmatrix}=0,
		\end{aligned}
	\end{equation*}
	for every $z=x+yi$ in $B_{\mathbb{C}}(\gamma(1),r_1)$ and for all $\ell\in{1,...,n}$. By the definition in \eqref{eq-holo def}, it follows that $F_{\Omega_1}^f$ is holomorphic in $\mathscr{P}(\mathbb{C}^n,\Omega)$.
\end{proof}

\section{Star-product of slice regular functions}\label{sec-star product}

In this section, we show that the $*$-product of slice regular functions is still slice regular. It implies that the class of slice regular functions is an associative unitary real algebra.
For this purpose, we will require the assistance of several results.

Let $\Omega\subset\mathbb{H}_s^n$, $\gamma\in\mathscr{P}(\mathbb{C}^n,\Omega)$ and $I\in\mathbb{S}(\Omega,\gamma)$. Denote
\begin{equation*}
	r_{\gamma,\Omega}^{I}:=\sup\left\{r\in[0,+\infty):B_I\left(\gamma^I(1),r\right)\subset\Omega_I\right\}.
\end{equation*}
Geometrically, it represents the largest radius $r$ for which a ball $B_I(\gamma^I(1), r)$, centered at the point $\gamma^I(1)$ in the slice $I$ of $\Omega$, remains entirely inside the subset $\Omega_I$ of $\Omega$ associated with the slice $I$. In simpler terms, it is the largest ball  you can draw around the endpoint of the path $\gamma^I$   without leaving the space $\Omega_I$.

\begin{lem}
	Let $\Omega\in\tau_s(\mathbb{H}_s^n)$, $\gamma\in\mathscr{P}(\mathbb{C}^n,\Omega)$ and $I\in\mathbb{S}(\Omega,\gamma)$. It holds that  $$r_{\gamma,\Omega}^{I}>0.$$ 
	Furthermore, for any point $z$ within the complex ball $B_{\mathbb{C}}(\gamma(1),r_{\gamma,\Omega}^{I})$,
 we have 
	\begin{equation}\label{eq-Lgammaz}
		\mathscr{L}_{\gamma}(z)\in\mathscr{P}(\mathbb{C}^n,\Omega).
	\end{equation}
\end{lem}

\begin{proof}
	Given that the point $\gamma^I(1)$ is contained in $\Omega$ and $\Omega_I$ is a member of the topology $\tau(\mathbb{C}_I^n)$, there exists a positive radius $r$ satisfying
	\begin{equation*}
		B_I\left(\gamma^I(1),r\right) \subset  \Omega_I.
	\end{equation*}
	This observation ensures that $r_{\gamma,\Omega}^{I}$ is greater than zero.

From its definition, it follows that
\begin{equation*}
	B_I\left(\gamma^I(1),r'\right) \subseteq \Omega_I,
\end{equation*}
 for all  $r' \in \left(0, r_{\gamma,\Omega}^{I}\right)$.
Consequently, we can assert that
\begin{equation}\label{eq-BI in OmegaI}
	B_I\left(\gamma^I(1), r_{\gamma,\Omega}^{I}\right) = \bigcup_{r' \in \left(0, r_{\gamma,\Omega}^{I}\right)} B_I\left(\gamma^I(1), r'\right) \subset  \Omega_I.
\end{equation}
For any point $z$ in the complex ball $B_{\mathbb{C}}\left(\gamma(1), r_{\gamma,\Omega}^{I}\right)$ and considering equation \eqref{eq-BI in OmegaI}, we have that $\left(\mathcal{L}_{\gamma(1)}^z\right)^I$ is a subset of $\Omega_I$. This implies
\begin{equation*}
	\left(\mathscr{L}_{\gamma}(z)\right)^I = \left(\gamma \circ \mathcal{L}_{\gamma(1)}^z\right)^I = \gamma^I \circ \left(\mathcal{L}_{\gamma(1)}^z\right)^I \subset  \Omega_I.
\end{equation*}
Hence, equation \eqref{eq-Lgammaz} is satisfied, completing the proof.
\end{proof}

\begin{prop}
	Let $\Omega_1\in\tau_s(\mathbb{H}_s^n)$ be real-path-connected, $\Omega_2\in\tau_s(\mathbb{H}_s^n)$ be $\Omega_1$-stem-preserving, $f\in\mathcal{PS}(\Omega_1)$ and $g\in\mathcal{PS}(\Omega_2)$. 
Then, the following relationship holds:  \begin{equation}\label{eq-f*g to z}
		f*g\left(z^I\right)=\left(f\left(z^I\right),If\left(z^I\right)\right) \left(F_{\Omega_1}^g\circ\mathscr{L}_\gamma\right)(z),
	\end{equation}
	for every path  $\gamma\in\mathscr{P}(\mathbb{C}^n,\Omega_1)$, each element $I\in\mathbb{S}(\Omega_1,\gamma)$, any radius $r\in\left(0,r_{\gamma,\Omega_1}^{I}\right)$ and for all points $z\in B_\mathbb{C}\left(\gamma(1),r\right)$.
\end{prop}

\begin{proof}
	Consider $\gamma$ as a path in $\mathscr{P}(\mathbb{C}^n,\Omega_1)$, $I$ as an element of $\mathbb{S}(\Omega_1,\gamma)$, $r$ within the range $\left(0,r_{\gamma,\Omega_1}^{I}\right)$, and let $z$ be a point in the complex ball $B_{\mathbb{C}}\left(\gamma(1),r\right)$. From equation \eqref{eq-Lgammaz}, we have $\mathscr{L}_{\gamma}(z)$ residing in $\mathscr{P}(\mathbb{C}^n,\Omega_1)$. Referring to equations \eqref{eq-starproduct def} and \eqref{eq-overline}, the computation proceeds as follows:
	\begin{equation*}
	\begin{split}
		f*g(z^I)
		=&\left(f(z^I),\mathfrak{I}(z^I)f(z^I)\right) \mathscr{F}_{\Omega_1}^{g}(z^I)
		\\=&\left(f(z^I),If(z^I)\right) F_{\Omega_1}^g\left(\gamma\circ\mathcal{L}_{\gamma(1)}^z\right)
		\\=&\left(f(z^I),If(z^I)\right) F_{\Omega_1}^g\left(\mathscr{L}_{\gamma}(z)\right)
		\\=& \left(f\left(z^I\right),If\left(z^I\right)\right) \left(F_{\Omega_1}^g\circ\mathscr{L}_\gamma\right)(z).
	\end{split}
\end{equation*}
	This completes the proof.
\end{proof}

Let $c\in\mathbb{H}$ and $I\in\mathbb{S}$. Then
\begin{equation}\label{eq-icic}
	I(c,Ic)=(Ic,-c)=(c,Ic)\begin{pmatrix} & -1\\ 1 \end{pmatrix}=(c,Ic)\sigma.
\end{equation}

\begin{thm}\label{thm-sr}
	Let $\Omega_1\in\tau_s(\mathbb{H}_s^n)$ be real-path-connected, $\Omega_2\in\tau_s(\mathbb{H}_s^n)$ be $\Omega_1$-stem-preserving, $f\in\mathcal{SR}(\Omega_1)$ and $g\in\mathcal{SR}(\Omega_2)$. Then $f*g\in \mathcal{SR}(\Omega_1)$.
\end{thm}

\begin{proof}
	Consider $I$ as an element of $\mathbb{S}$ and let $q$ be a point in $\left(\Omega_1\right)_I$. There exists a complex number $z$ such that $q$ can be represented as $z^I$. Given that $\Omega_1$ is real-path-connected, we can find a path $\gamma$ in $\mathscr{P}(\mathbb{C}^n,\Omega_1)$ with $I$ in $\mathbb{S}(\Omega_1,\gamma)$, such that $\gamma^I$ lies in $\Omega_1$ and $q$ equals $\gamma^I(1)$. Referring to equations \eqref{eq-f*g to z}, \eqref{eq-icic}, and Proposition \ref{pr-FOmega1f holo}, the following derivation is obtained:
	\begin{equation*}
		\begin{aligned}
			&\frac{1}{2}\left(\frac{\partial}{\partial x\ell}+I\frac{\partial}{\partial y_\ell}\right)(fg)(z^I)
			\\=&\frac{1}{2}\left(\frac{\partial}{\partial x_\ell}+I\frac{\partial}{\partial y_\ell}\right)\left(f\left(z^I\right),If\left(z^I\right)\right) \left(F_{\Omega_1}^g\circ\mathscr{L}\gamma\right)(z)
			\\=&\left(\frac{1}{2}\left(\frac{\partial}{\partial x_\ell}+I\frac{\partial}{\partial y_\ell}\right)f\left(z^I\right),I\left[\frac{1}{2}\left(\frac{\partial}{\partial x_\ell}+I\frac{\partial}{\partial y_\ell}\right)\right]f\left(z^I\right)\right) \left(F_{\Omega_1}^g\circ\mathscr{L}\gamma\right)(z)
			\\ &+\left(f\left(z^I\right),If\left(z^I\right)\right)\left[\frac{1}{2}\left(\frac{\partial}{\partial x_\ell}+\sigma\frac{\partial}{\partial y_\ell}\right)\right]\left(F_{\Omega_1}^g\circ\mathscr{L}\gamma\right)(z)
			\\=&(0,I\cdot 0)\left(F_{\Omega_1}^g\circ\mathscr{L}_\gamma\right)(z)+\left(f\left(z^I\right),If\left(z^I\right)\right)\cdot 0=0.
		\end{aligned}
	\end{equation*}
	This result implies that $(fg)_I$ is holomorphic for every $I$ in $\mathbb{S}$. Consequently, it is established that $fg$ is slice regular.
\end{proof}

\begin{theorem}
	Suppose $\Omega$ is a subset of $\mathbb{H}_s^n$ that exhibits self-stem-preserving properties. In such a case, the structure $(\mathcal{SR}(\Omega), +, *)$ forms an associative unitary real algebra.
\end{theorem}

\begin{proof}
	The assertion of this theorem follows immediately from the application of Theorem \ref{tm-aura} in conjunction with Theorem \ref{thm-sr}.
\end{proof}

 In conclusion, our study has established that the $*$-product preserves the slice regularity of functions, thereby confirming that this class of functions forms an associative unitary real algebra. Our investigation into the star-product of slice regular functions over several  quaternionic variables has enriched our comprehension of their algebraic and geometric properties. Moreover, it has opened avenues for future research and potential applications.

\bibliographystyle{plain}
\bibliography{mybibfile}

\end{document}